\documentclass[a4paper,10pt]{article}
\pdfoutput=1
\usepackage{hyperref}
\hypersetup{hypertexnames = false, bookmarksdepth = 2, bookmarksopen = true, colorlinks, linkcolor = black, citecolor = black, urlcolor = black, pdfstartview={XYZ null null 1}}

\usepackage{amsfonts}
\usepackage[fleqn, leqno]{amsmath}
\usepackage{amsthm}

\usepackage[capitalise,noabbrev]{cleveref}

\usepackage{booktabs}
\usepackage{colonequals}
\usepackage{diagbox}
\usepackage{enumitem}
\usepackage{fullpage}
\usepackage{mathtools}
\usepackage{multirow}
\usepackage{parskip}
\usepackage{thmtools}
\usepackage{tikz-cd}
\usepackage[colorinlistoftodos, textsize = footnotesize]{todonotes}
\usepackage{xparse}
\usepackage{xpatch}
\usepackage{xspace}

\usepackage[utf8]{inputenc}
\usepackage[T1]{fontenc}
\usepackage{libertine}
\usepackage[libertine]{newtxmath}
\usepackage[scaled=0.83]{beramono}
\usepackage{eucal}
\usepackage{microtype}
\usepackage{gitinfo2}

\usepackage[backend=biber, maxbibnames=99, datamodel=mrnumber, sortcites]{biblatex}

\DeclareFieldFormat{mrnumber}{%
  MR\addcolon\space
  \ifhyperref
    {\href{http://www.ams.org/mathscinet-getitem?mr=#1}{\nolinkurl{#1}}}
    {\nolinkurl{#1}}}

\renewbibmacro*{doi+eprint+url}{%
  \iftoggle{bbx:doi}
    {\printfield{doi}}
    {}%
  \newunit\newblock
  \printfield{mrnumber}%
  \newunit\newblock
  \iftoggle{bbx:eprint}
    {\usebibmacro{eprint}}
    {}%
  \newunit\newblock
  \iftoggle{bbx:url}
    {\usebibmacro{url+urldate}}
    {}}

\DeclareSourcemap{
  \maps[datatype=bibtex]{
    \map{
      \step[fieldset=issn, null]
    }
  }
}
\xpretobibmacro{bib+doi+url}
  {\iffieldundef{doi}
     {}
     {\clearfield{url}}}
  {}{}

\xpretobibmacro{cite+doi+url}
  {\iffieldundef{doi}
     {}
     {\clearfield{url}}}
  {}{}

\newcounter{todocounter}
\DeclareDocumentCommand\addreference{g}{\stepcounter{todocounter}\todo[color = blue!30]{\thetodocounter. Add reference\IfNoValueF{#1}{: #1}}\xspace}
\DeclareDocumentCommand\checkthis{g}{\stepcounter{todocounter}\todo[color = red!50]{\thetodocounter. Check this\IfNoValueF{#1}{: #1}}\xspace}
\DeclareDocumentCommand\fixthis{g}{\stepcounter{todocounter}\todo[color = orange!50]{\thetodocounter. Fix this\IfNoValueF{#1}{: #1}}\xspace}
\DeclareDocumentCommand\expand{g}{\stepcounter{todocounter}\todo[color = green!50]{\thetodocounter. Expand\IfNoValueF{#1}{: #1}}\xspace}

\makeatletter
\renewcommand\thmt@autorefsetup{%
  \@xa\def\csname\thmt@envname autorefname\@xa\endcsname\@xa{\thmt@thmname}%
}
\makeatother

\declaretheoremstyle[
  spaceabove = 3pt,
  spacebelow = 3pt,
  bodyfont = \itshape,
]{first}
\declaretheoremstyle[
  spaceabove = 3pt,
  spacebelow = 3pt,
]{second}
\declaretheorem[numberwithin=section, style=first]{theorem}

\declaretheorem[sibling=theorem, style=first]{lemma}
\declaretheorem[sibling=theorem, style=first]{proposition}

\declaretheorem[sibling=theorem, style=second]{example}
\declaretheorem[sibling=theorem, style=second]{remark}

\Crefname{assumption}{Assumption}{Assumptions}
\Crefname{convention}{Convention}{Conventions}
\Crefname{setup}{Setup}{Setups}

\declaretheorem[numberwithin=section, style=first, title=Theorem]{alphatheorem}

\crefname{alphatheorem}{Theorem}{Theorems}
\crefname{alphaconjecture}{Conjecture}{Conjectures}
\crefname{alphacorollary}{Corollary}{Corollaries}
\crefname{alphaproposition}{Proposition}{Propositions}

\mathchardef\mhyphen="2D

\DeclareMathOperator\Rep{R}
\DeclareDocumentCommand\modulistack{om}{\IfNoValueTF{#1}{\mathcal{M}{(#2)}}{\mathcal{M}^{#1}(#2)}}
\DeclareDocumentCommand\modulispace{om}{\IfNoValueTF{#1}{\mathrm{M}{(#2)}}{\mathrm{M}^{#1}(#2)}}
\DeclareDocumentCommand\repspace{om}{\IfNoValueTF{#1}{\mathrm{R}{(#2)}}{\mathrm{R}^{#1}(#2)}}
\newcommand\moduli{\ensuremath{\mathrm{M}}}

\DeclareMathOperator\source{s}
\DeclareMathOperator\target{t}

\newcommand\semistable{\ensuremath{\mhyphen\mathrm{sst}}}
\newcommand\stable{\ensuremath{\mhyphen\mathrm{st}}}
\newcommand\semisimple{\ensuremath{\mathrm{ssimp}}}

\newcommand\et{\ensuremath{\textrm{\'et}}}
\newcommand\field{\ensuremath{\mathbf{k}}}
\newcommand\Gm{\ensuremath{\mathbb{G}_{\mathrm{m}}}}
\newcommand\rig{\ensuremath{\mathrm{rig}}}
\newcommand\ur{\ensuremath{\mathrm{ur}}}
\newcommand\smooth{\ensuremath{\mathrm{sm}}}

\newcommand\vect[1]{\ensuremath{\boldsymbol #1}}

\DeclareMathOperator\Br{Br}
\DeclareMathOperator\dimvect{\underline{dim}}
\DeclareMathOperator\Ext{Ext}
\DeclareMathOperator\ext{ext}
\DeclareMathOperator\GL{GL}
\DeclareMathOperator\HH{H}
\DeclareMathOperator\image{im}
\DeclareMathOperator\Hom{Hom}
\DeclareMathOperator\Mat{Mat}
\DeclareMathOperator\PGL{PGL}
\DeclareMathOperator\Pic{Pic}
\DeclareMathOperator\sheafEnd{\mathcal{E}nd}
\DeclareMathOperator\Spec{Spec}

\usepackage{xpatch}

\makeatletter
\patchcmd\blx@bblinput{\blx@blxinit}
                      {\blx@blxinit
                      }{}{\fail}
\makeatother

\title{Brauer groups of resolved quiver moduli via gerbes}
\author{Pieter Belmans \and Gianni Petrella \and Sebasti\'an Torres}

\begin{document}
\maketitle

\begin{abstract}
  We show that the Brauer group of any resolution of singularities of
  the moduli space of semistable quiver representations
  is trivial.
  We do this by extending the quiver-curve dictionary,
  translating a proof of the analogous result by Biswas--Hogadi--Holla
  for moduli of vector bundles on a curve
  to the setting of moduli of quiver representations,
  giving an algebro-geometric proof.
  This gives a new proof of this triviality,
  first proved by Le Bruyn--Schofield,
  building on algebraic (resp.~cohomological)
  vanishing results due to Saltman (resp.~Colliot-Th\'el\`ene--Sansuc).
  Reversing the logic,
  our approach gives
  a new algebro-geometric proof of these vanishing results.

\end{abstract}

\section{Introduction}
Moduli spaces of (semi)stable quiver representations
have been introduced in the 1990s \cite{MR1315461}.
Moduli spaces of (semi)stable vector bundles on curves
have been introduced in the 1960s \cite{MR0233371,MR0242185,MR0184252,MR0699278}.
Whilst their origins are seemingly very distinct,
the former coming from representation theory
and the latter from algebraic geometry,
they are both associated to categories of homological dimension~1.
This leads to interesting parallels between the two settings.

The starting point for this work
is the following result \cite[Theorem~1.1]{MR2925603},
for moduli spaces of vector bundles on curves,
where we work over an algebraically closed field~$\field$ of characteristic 0.
\begin{theorem}[Biswas--Hogadi--Holla]
  \label{theorem:biswas-hogadi-holla}
  Let~$C$ be a smooth projective curve of genus~$g\geq 2$.
  Let~$r\geq 2$,
  and~$\mathcal{L}\in\Pic C$.
  Let~$M=\moduli_C(r,\mathcal{L})$ be
  the moduli space of semistable vector bundles of rank~$r$
  and determinant~$\mathcal{L}$.
  For every resolution of singularities~$\widetilde{M}$ of~$M$,
  the Brauer group of $\widetilde{M}$ vanishes.
\end{theorem}
\vspace{-.5\baselineskip}

The quiver moduli analogue of \cref{theorem:biswas-hogadi-holla},
which is the main result of this article,
reads as follows.
\begin{alphatheorem}
  \label{theorem:main}
  Let~$Q$ be an acyclic quiver,
  let~$\vect{d}$ be a dimension vector,
  and let~$\theta$ be a stability parameter
  for which~$\vect{d}$ is amply stable.
  Let~$M=\modulispace[\theta\semistable]{Q,\vect{d}}$ be
  the moduli space of~$\theta$-semistable representations
  of dimension vector~$\vect{d}$.
  For every resolution of singularities~$\widetilde{M}$ of~$M$,
  the Brauer group of $\widetilde{M}$ vanishes.
\end{alphatheorem}
\vspace{-.5\baselineskip}
This result is not actually new.
It can be seen to be equivalent to Le Bruyn--Schofield's \cite[Corollary~5~(3)]{MR1048393}.
Their corollary is proved using
\begin{itemize}
  \item the fact that the stable birational equivalence class of
    a quiver moduli space of semisimple representations
    is determined by~$\gcd(\vect{d})$ \cite[Theorem~3]{MR1048393}
    (see also \cref{proposition:function-fields}, which incorporates semistability,
    a notion not yet introduced at the time of op.~cit.), and
  \item a deep result of Saltman \cite[Theorem~2.9]{MR0812169},
    followed shortly afterwards by Colliot-Th\'el\`ene--Sansuc's proof
    using different techniques \cite[Theorem~9.7]{MR0878473},
    on the vanishing of the unramified Brauer group of the centre~$\mathrm{C}_n$ of generic~$n\times n$ matrices
    (see \cref{subsection:saltman} for more on this object and its role in quiver moduli).
\end{itemize}
An overview of the situation is provided by \cref{table:summary}.
In \cref{remark:kronecker-to-jordan}
we explain how to go directly
(avoiding Le Bruyn--Schofield's stable birational equivalence)
from~$\Br(\widetilde{M})$ for
an appropriate choice of acyclic quiver
to~$\Br_\ur(\mathrm{C}_n)$,
which concerns the (non-acyclic) generalised Jordan quiver.

Reversing the logic,
one can also consider \cref{theorem:main}
in combination with the Le Bruyn--Schofield stable birational equivalence
to give an algebro-geometric (as opposed to algebraic, resp.~cohomological)
proof of Saltman's resp.~Colliot-Th\'el\`ene--Sansuc's result,
at least in the case where~$\field$ is algebraically closed and of characteristic~0.

\begin{table}[h]
  \centering
  \begin{tabular}{lclcl}
    & \clap{Le Bruyn--Schofield} \\
    & \clap{\cite[Theorem~3]{MR1048393}}
    &
    & \clap{\cref{proposition:function-fields}} \\
    \multicolumn{1}{c}{$\Br_\ur(\mathrm{C}_n)=0\qquad$}                                         &
    \multicolumn{1}{c}{$\Leftrightarrow$}                                                       &
    \multicolumn{1}{c}{$\qquad\Br_\ur(\field(\Rep(Q,\vect{d}))^{\GL_{\vect{d}}})=0\qquad$}      &
    \multicolumn{1}{c}{$\Leftrightarrow$}                                                       &
    \multicolumn{1}{c}{$\qquad\Br(\widetilde{M})=0$} \\
    \midrule
    \multicolumn{4}{l}{\textbullet\ Saltman: \cite[Theorem~2.9]{MR0812169}}                     & \textbullet\ Le Bruyn--Schofield: \\
    \phantom{\textbullet} \emph{algebraic}                                                      &                                              &  &  & \cite[Corollary~5~(3)]{MR1048393} \\
    \multicolumn{4}{l}{\textbullet\ Colliot-Th\'el\`ene--Sansuc: \cite[Theorem~9.7]{MR0878473}} & \textbullet\ \cref{theorem:main} \\
    \phantom{\textbullet} \emph{cohomological}                                                  &                                              &  &  & \phantom{\textbullet}\emph{algebro-geometric} \\
  \end{tabular}
  \caption{Summary table}
  \label{table:summary}
\end{table}

\paragraph{Quiver-curve dictionary}
The goal of this paper is not to prove a new result
(because it was known already by Le Bruyn--Schofield)
but to expand the quiver-curve dictionary.
Namely, we will take our approach for \cref{theorem:biswas-hogadi-holla}
from \cite{MR2925603}
for moduli spaces of vector bundles on curves,
and explain how to adapt it to the setting of quiver moduli.
This gives a new instance of the quiver-curve dictionary,
relating the properties (and their proofs) of
moduli of vector bundles on curves
and moduli of quiver representations.

Other instances of the many parallels between
these moduli spaces
are given by similarities in
\begin{itemize}[noitemsep]
  \item their constructions as quotients using geometric invariant theory,
    their hyperk\"ahler counterparts, and the role of the Weil conjectures in the study of both,
    surveyed in \cite{MR3882963};
  \item the role Lie algebras play,
    surveyed in \cite{MR3966814};
  \item automorphisms and deformations of moduli spaces,
    comparing \cite{MR0384797} for curves
    to \cite{MR4954467,MR4883104} for quivers;
  \item the universal object giving a fully faithful Fourier--Mukai functor,
    comparing \cite{MR3954042,MR3713871,MR3764066} for curves
    to \cite{MR4883104,MR1688469} for quivers;
  \item the construction of intrinsic ample line bundles on the good moduli spaces
    from the stacky description
    without resorting to geometric invariant theory,
    comparing \cite{MR4480534} for curves
    to \cite{MR5066465} for quivers;
  \item their birational properties,
    comparing \cite{MR1820549} for curves
    to \cite{MR1914089} for quivers;
  \item the positivity properties for the anticanonical line bundle,
    with moduli of vector bundles being Fano varieties \cite{MR0999313}
    and quiver moduli often being Fano varieties \cite{MR4352662}.
\end{itemize}
This article should thus mostly be seen
in light of further extending the quiver-curve dictionary,
and as an elaboration of the special and interesting features
for quiver moduli.

In \cref{section:quiver-moduli} we recall the relevant details in the construction of quiver moduli,
prove some results on
how function fields of quiver moduli compare to one another (\cref{proposition:function-fields}),
on the existence of stable representations (\cref{proposition:fractions}),
and recall the Reineke--Schr\"oer description of the Brauer group of quiver moduli from \cite{MR3683503},
which uses geometric objects called smooth models, see \cite{MR2511752}.
In \cref{section:main} we discuss how one can translate
the approach from \cite{MR2925603}
to the setting at hand.
In \cref{section:remarks}
we give some remarks on \cref{theorem:biswas-hogadi-holla} and \cref{theorem:main},
we describe for an algebro-geometric audience what the role of generic matrices is in the earlier results,
and we comment on some loose ends in \cite{MR2925603,MR3683503}.

\paragraph{Acknowledgements}
We would like to thank Markus Reineke for interesting discussions.
P.B.~was partially supported by the Luxembourg National Research Fund (FNR--17113194) and by NWO (\href{https://doi.org/10.61686/RZKLF82806}{\texttt{doi:10.61686/RZKLF82806}}).
G.P.~was supported by the Luxembourg National Research Fund (FNR--17953441).
S.T.~was supported by the Luxembourg National Research Fund (FNR--17113194) and by ANID, Fondecyt Postdoctoral Grant 3240013.

\section{Quiver moduli}
\label{section:quiver-moduli}
\subsection{Recollection of the construction}
\label{subsection:recollection}
Throughout we let~$\field$ denote an algebraically closed field of characteristic~0.

We let~$Q=(Q_0,Q_1)$ be a quiver,
with vertices~$Q_0$ and arrows~$Q_1$.
An arrow~$\alpha\in Q_1$ has
a source~$\source(\alpha)$ and target~$\target(\alpha)$ in~$Q_0$.
A representation~$V$ of~$Q$ is the data of
a finite-dimensional vector space~$V_i$ for every~$i\in Q_0$
and a linear morphism~$V(\alpha)\colon V_{\source(\alpha)}\to V_{\target(\alpha)}$
for every~$\alpha\in Q_1$.

After choosing bases for the vector spaces~$V_i$,
the affine space
\begin{equation}
  \Rep(Q,\vect{d})\colonequals\prod_{\alpha\in Q_1}\Mat_{d_{\target(\alpha)},d_{\source(\alpha)}}(\field)
\end{equation}
encodes all representations of dimension vector~$\vect{d}$.
The reductive algebraic group
\begin{equation}
  \GL_{\vect{d}}\colonequals\prod_{i\in Q_0}\mathrm{GL}_{d_i}(\field)
\end{equation}
acts on~$\Rep(Q,\vect{d})$ via change of basis,
encoded by conjugation:~$g=(g_i)_{i\in Q_0}$
acts on the tuple~$(V(\alpha))_{\alpha\in Q_1}$ by
\begin{equation}
  g\cdot(V(\alpha))_{\alpha\in Q_1}\colonequals(g_{\target(\alpha)}V(\alpha) g_{\source(\alpha)}^{-1})_{\alpha\in Q_1}.
\end{equation}
The orbits of this action correspond to isomorphism classes of representations.
The central closed subgroup
\begin{equation}
  \Delta\colonequals\{(z\cdot\mathrm{I}_{d_i})_{i\in Q_0}\mid z\in\Gm\}
\end{equation}
of~$\GL_{\vect{d}}$ acts trivially,
so we also have an action by
\begin{equation}
  \PGL_{\vect{d}}\colonequals\GL_{\vect{d}}/\Delta.
\end{equation}
The moduli stack~$[\Rep(Q,\vect{d})/\GL_{\vect{d}}]$
has as its good moduli space~$\Spec\field[\Rep(Q,\vect{d})]^{\GL_{\vect{d}}}$,
which is the moduli space $\modulispace[\semisimple]{Q,\vect{d}}$ of semisimple representations.
When~$Q$ is an acyclic quiver,
this is just a point.

To improve the geometry of the moduli spaces,
one introduces a stability parameter~$\theta\in\Hom(\mathbb{Z}^{Q_0},\mathbb{Z})$.
Assume moreover that~$\theta(\vect{d})=0$.
Then, a representation~$V$ is said to be~$\theta$-stable
(resp.~$\theta$-semistable)
if for every proper non-zero subrepresentation~$W$ of~$V$
we have~$\theta(\dimvect W)<0$
(resp.~$\theta(\dimvect W)\leq 0$).
This induces Zariski-open subsets
\begin{equation}
  \label{equation:inclusions-Rep}
  \Rep^{\theta\stable}(Q,\vect{d})
  \subseteq
  \Rep^{\theta\semistable}(Q,\vect{d})
  \subseteq
  \Rep(Q,\vect{d}),
\end{equation}
and by applying geometric invariant theory we obtain
an open immersion
(resp.~a projective morphism to an affine variety)
\begin{equation}
  \label{equation:3-moduli-spaces}
  \modulispace[\theta\stable]{Q,\vect{d}}
  \hookrightarrow
  \modulispace[\theta\semistable]{Q,\vect{d}}
  \twoheadrightarrow
  \modulispace[\semisimple]{Q,\vect{d}}.
\end{equation}
In particular, when~$Q$ is acyclic,~$\modulispace[\theta\semistable]{Q,\vect{d}}$
is a projective variety.
This will be the case we consider in \cref{section:main}.
Because the global dimension of~$\operatorname{rep}Q$ is~1,
we have that the stable locus is contained in the smooth locus:
\begin{equation}
  \modulispace[\theta\stable]{Q,\vect{d}}
  \subseteq
  \modulispace[\theta\semistable]{Q,\vect{d}}_{\mathrm{sm}}.
\end{equation}

\paragraph{Comparing function fields of quiver moduli}
For the translation between different approaches and results,
which often predate the introduction of stability,
we need to know how function fields of quiver moduli compare.
Recall that~$\vect{d}$ is said to be a \emph{Schur root}
if there exists a representation of dimension vector~$\vect{d}$
whose endomorphism algebra is~$\field$,
or equivalently,
that the general representation of dimension vector~$\vect{d}$
is indecomposable.
Recall also that
$\vect{d}$ is a Schur root
if and only if
there exists a stability parameter~$\theta$
for which there exists a~$\theta$-stable representation
of dimension vector~$\vect{d}$
\cite[Proposition~4.4]{MR1315461}.

The following result is completely standard,
but we have not found a reference in the literature stating it.
\begin{proposition}
  \label{proposition:function-fields}
  Let~$Q$ be a quiver,
  and let~$\vect{d}$ be a Schur root.
  Then for any stability parameter~$\theta$
  for which there exists a~$\theta$-stable representation
  of dimension vector~$\vect{d}$,
  there exist isomorphisms
  \begin{equation}
    \field(\Rep(Q,\vect{d}))^{\GL_{\vect{d}}}
    \cong
    \field(\Rep(Q,\vect{d}))^{\PGL_{\vect{d}}}
    \cong
    \field(\modulispace[\theta\stable]{Q,\vect{d}})
    \cong
    \field(\modulispace[\theta\semistable]{Q,\vect{d}}).
  \end{equation}
\end{proposition}
This allows us to interpret the results from \cite{MR1048393},
predating \cite{MR1315461},
in our setup.

\begin{proof}
  The first isomorphism exists by the definition of~$\PGL_{\vect{d}}$.
  The third isomorphism follows by the open immersion in \eqref{equation:3-moduli-spaces},
  where both varieties are non-empty
  by the assumption on~$\theta$.

  For the second isomorphism,
  one uses the theory of \emph{rational quotients} due to Rosenlicht.
  By \cite[Theorem~2]{MR0082183}
  (see, e.g., \cite[Theorem~1.1]{MR3676045} for an equivalent modern statement,
  and \cite[\S2]{MR1100485} for more context)
  there exists a~$\GL_{\vect{d}}$-invariant dense open subset~$U\subset\Rep(Q,\vect{d})$
  and a dominant morphism~$\pi\colon U\to Z$
  such that for all~$x\in U$ we have~$\pi^{-1}(\pi(x))=\GL_{\vect{d}}\cdot x$,
  and~$\field(\Rep(Q,\vect{d}))^{\GL_{\vect{d}}}\cong\field(Z)$.
  This rational quotient is also universal with respect to rational maps
  (i.e., it is a categorical quotient with respect to rational maps):
  if~$Z'$ is another such quotient,
  for some~$\GL_{\vect{d}}$-invariant dense open subset~$U'$,
  then there exists an induced birational equivalence~$Z\dashrightarrow Z'$,
  inducing an isomorphism~$\field(Z)\cong\field(Z')$.
  Taking~$Z'=\modulispace[\theta\stable]{Q,\vect{d}}$
  gives the second isomorphism.
\end{proof}

Note that it is possible that~$\field(\Rep(Q,\vect{d}))^{\GL_{\vect{d}}}$
is strictly larger than~$\field(\modulispace[\semisimple]{Q,\vect{d}})=\operatorname{Frac}(\field[\Rep(Q,\vect{d})]^{\GL_{\vect{d}}})$:
simply consider an acyclic quiver~$Q$
and suitable~$\vect{d}$ and~$\theta$,
for which the latter is~$\field$,
whereas the former is the function field of
a (non-trivial) moduli space of~$\theta$-(semi)stable representations
by \cref{proposition:function-fields}.

\subsection{Existence of stable representations}
\label{subsection:existence-stable}
At some point in the proof we will need to construct a~$\theta$-stable representation
with (indivisible) dimension vector~$\vect{d}/\gcd(\vect{d})$.
The analogous existence is automatic in the case of vector bundles,
as there exists a stable vector bundle for all pairs~$(r,d)$
with~$r\geq 1$,
see, e.g., \cite[Theorem~8.6.1]{MR1428426}.
However, for moduli of quiver representations this is not automatic,
as the moduli space can be empty,
even for indivisible dimension vectors\footnote{
  It suffices to consider the Kronecker quiver~$\mathrm{K}_2$
  and dimension vector~$\vect{d}=(1,1)$,
  and take for~$\theta$ the inner product with~$(-1,1)$,
  so that the simple subrepresentation~$\mathrm{S}_2$ is always destabilising.
  On the other side of the wall given by the zero morphism
  we find that~$\modulispace[\theta\stable]{\mathrm{K}_2,(1,1)}\cong\mathbb{P}^1$.
}.

What we thus need in \cref{section:main} is (the second point of) the following proposition.
\begin{proposition}
  \label{proposition:fractions}
  Let~$Q$ be a quiver,
  let~$\vect{d}$ be a dimension vector,
  and let~$n\geq 2$.
  \begin{enumerate}
    \item\label{enumerate:schur-root-fraction}
      If~$n\vect{d}$ is a Schur root,
      then~$\vect{d}$ is a Schur root.
    \item\label{enumerate:stable-fraction}
      If~$n\vect{d}$ admits a~$\theta$-stable representation,
      then~$\vect{d}$ admits a~$\theta$-stable representation.
  \end{enumerate}
\end{proposition}
This result seems to be missing from the literature,
which is why we include a proof here.
In \cite[Theorem~3.7]{MR1162487}
it is shown that integer multiples of imaginary Schur roots,
i.e., Schur roots for which~$\langle\vect{d},\vect{d}\rangle\leq 0$,
or equivalently,
roots for which there exist infinitely many non-isomorphic representations,
are again imaginary Schur roots\footnote{This is not true for real Schur roots.}.
The converse, that fractions of Schur roots are again Schur roots,
is however missing from op.~cit.
The result on~$\theta$-stable representations seems to be new.

We will use Schofield's criterion for Schur roots \cite[Theorem~6.1]{MR1162487}.
Recall that a subdimension vector~$\vect{e}\leq\vect{d}$ is \emph{general}
if the subset
\begin{equation}
  \{V\in\Rep(Q,\vect{d})\mid\exists V'\subseteq V\colon\dimvect V'=\vect{e}\}\subseteq\Rep(Q,\vect{d})
\end{equation}
is non-empty and Zariski-open.

\begin{proposition}[Schofield's criterion for Schur roots]
  \label{proposition:schofield-criterion}
  Let~$Q$ be a quiver.
  Let~$\vect{d}$ be a dimension vector.
  Then~$\vect{d}$ is a Schur root
  if and only if
  for every proper non-zero general subdimension vector~$\vect{e}$
  we have that~$\langle\vect{e},\vect{d}\rangle-\langle\vect{d},\vect{e}\rangle>0$.
\end{proposition}

We will also use the following criterion
for admitting $\theta$-stable representations.
Although it is certainly not new,
we include a proof of this criterion below,
as it does not seem to appear in the literature.

\begin{proposition}
  \label{proposition:theta-stable-criterion}
  Let~$Q$ be a quiver.
  Let~$\vect{d}$ be a dimension vector,
  and let~$\theta$ be a stability parameter
  for which~$\theta(\vect{d})=0$.
  Then there exists a~$\theta$-stable representation of dimension vector~$\vect{d}$
  if and only if
  for every proper non-zero general subdimension vector~$\vect{e}$
  we have that~$\theta(\vect{e})<0$.
\end{proposition}
\begin{proof}
  It suffices to remark that the general representation
  in~$\Rep(Q, \vect{d})$ only admits subrepresentations
  whose dimension vector~$\vect{e}$ is a general subdimension
  vector of~$\vect{d}$.
  If no proper non-zero general subdimension vector $\vect{e}$ of $\vect{d}$
  satisfies $\theta(\vect{e}) \geq 0$, the general representation
  is thus stable.
  Conversely, if there is a stable representation,
  then the general one is also stable since the stable locus is open,
  so we must have~$\theta(\vect{e})<0$
  for every general subdimension vector $\vect{e}$ of $\vect{d}$
  by the definition of stability.
\end{proof}

Recall that $\hom(\vect{d},\vect{d}')$ and $\ext(\vect{d},\vect{d}')$ denote the generic values
of $\dim\Hom(V,V')$ and $\dim\Ext^1(V,V')$ for $V\in\Rep(Q,\vect{d})$ and $V'\in\Rep(Q,\vect{d}')$,
as introduced in~\cite[\S3]{MR1162487}.
We will also need the following lemma, which follows immediately
from the identity
\begin{equation}
  \hom(\vect{d}, \vect{d}') - \ext(\vect{d}, \vect{d}') = \langle \vect{d}, \vect{d}'\rangle,
\end{equation}
together with~\cite[Remark~5.2]{MR5007902}.

\begin{lemma}
  \label{lemma:general-hom-multiple}
  Let~$\hom(\vect{d}, \vect{d}')$ be the general~$\hom$ for dimension vectors~$\vect{d}, \vect{d}'$.
  Then, for any~$n\geq 1$,
  \begin{equation}
    \hom(n \vect{d}, n \vect{d}') = n^2 \hom(\vect{d}, \vect{d}').
  \end{equation}
\end{lemma}

We are now in a position to prove the claim of this subsection.

\begin{proof}[Proof of \cref{proposition:fractions}]
  Let~$\vect{e}$ be a proper non-zero general subdimension vector of~$\vect{d}$.
  We will first show that~$n\vect{e}$
  is a general subdimension vector of~$n\vect{d}$,
  and then we will use this to show that Schofield's criterion from \cref{proposition:schofield-criterion}
  for~$\vect{d}$ to be a Schur root,
  respectively the criterion from \cref{proposition:theta-stable-criterion}
  for~$\vect{d}$ to admit~$\theta$-stable representations,
  is satisfied.

  By \cite[Theorem~3.3]{MR1162487}
  we have that~$\vect{e}$ is a general subdimension vector of~$\vect{d}$
  if and only if~$\ext(\vect{e},\vect{d}-\vect{e})=0$.
  In this case, we have
  \begin{equation}
    \hom(n\vect{e},n(\vect{d} - \vect{e})) - \ext(n\vect{e}, n(\vect{d} - \vect{e}))
    = \langle n\vect{e}, n(\vect{d} - \vect{e}) \rangle
    = n^2 \langle \vect{e}, \vect{d} - \vect{e}\rangle
    = n^2  \hom(\vect{e}, \vect{d} - \vect{e}),
  \end{equation}
  where the last equality holds because we assumed that~$\vect{e}$
  is a general subdimension vector of~$\vect{d}$;
  to apply \cref{proposition:schofield-criterion}
  it is then enough to show that~$\hom(n\vect{e},n(\vect{d} - \vect{e})) = n^2\hom(\vect{e}, \vect{d} - \vect{e})$,
  which follows from \cref{lemma:general-hom-multiple}.

  As~$n\vect{e}$ is a general subdimension vector of~$n\vect{d}$,
  Schofield's criterion from \cref{proposition:schofield-criterion} gives
  \begin{equation}
    \langle n\vect{d}, n\vect{e}\rangle - \langle n\vect{e}, n\vect{d}\rangle < 0,
  \end{equation}
  which is equivalent to~$\langle \vect{d}, \vect{e} \rangle - \langle \vect{e}, \vect{d} \rangle<0$,
  thus another application of Schofield's criterion shows that~$\vect{d}$ is a Schur root.

  The stable version is analogous.
  Since $n\vect{d}$ admits a $\theta$-stable representation we have $\theta(n\vect{d})=0$,
  and hence $\theta(\vect{d})=0$ by linearity of $\theta$,
  so $\theta$ is a stability parameter for $\vect{d}$.
  For $\vect{e}$ a proper non-zero general subdimension vector of $\vect{d}$,
  we have already shown that $n\vect{e}$ is a general subdimension vector of $n\vect{d}$,
  so by \cref{proposition:theta-stable-criterion} applied to $n\vect{d}$
  we have $\theta(n\vect{e})<0$,
  and hence $\theta(\vect{e})<0$ by linearity of $\theta$.
  Applying \cref{proposition:theta-stable-criterion} again, this time to $\vect{d}$,
  we conclude that $\vect{d}$ admits a $\theta$-stable representation.
\end{proof}

\paragraph{Smooth models}
On a moduli space,
the (non-)existence of a universal bundle
is an important property,
as it can be used to probe the geometry of the moduli space.
In the setting of moduli of vector bundles on curves
this question is answered (in three different ways) in \cite{MR2353678,MR0999313,MR0325615},
and existence is shown to be equivalent to all semistable bundles being stable.
The analogous result for quiver moduli is shown in \cite[Theorem~4.4]{MR3683503},
using the Brauer group,
translating the method of proof from \cite{MR2353678}
to quiver moduli.

For moduli of vector bundles,
a universal projective bundle over the stable locus
can be constructed,
whose fibre at a point~$x$
is the projectivisation of
the stable vector bundle corresponding to~$x$ \cite{MR0325615}.
However, when there are strictly semistable points,
the projective bundle is not Zariski-locally trivial,
or equivalently,
there is no universal vector bundle on the stable locus
whose projectivisation is this projective bundle.

The analogous construction for quiver moduli is as follows.
Let~$\vect{n}$ be another dimension vector.
In \cite[\S3]{MR2511752},
a framing on~$Q$ (dependent on~$\vect{n}$) is used to define the \emph{smooth model}~$P_{\vect{n}}$.
Forgetting the framing data gives a projective morphism
\begin{equation}
  P_{\vect{n}}\to\modulispace[\theta\semistable]{Q,\vect{d}},
\end{equation}
which on the stable locus has as its fibres
a projective space of dimension~$\vect{n}\cdot\vect{d}-1$,
making it a Brauer--Severi scheme over the stable locus.
Moreover,
\cite[Proposition~3.8]{MR2511752} shows that
the smooth model~$P_{\vect{n}}$
is isomorphic to~$\mathbb{P}(\bigoplus_{i\in Q_0}\mathcal{U}_i^{\oplus n_i})$,
where~$\mathcal{U}_i$ denotes the universal locally free sheaf at the vertex~$i$
on the stable locus,
provided that~$\vect{d}$ is~$\theta$-coprime (and in particular coprime),
so that there are no strictly semistable points,
whereas without this condition the fibre over a stable point~$V$ is merely
the projectivisation of~$\bigoplus_{i\in Q_0}V_i^{\oplus n_i}$.

\paragraph{Brauer groups of quiver moduli}
We are now in a position to recall the main result of \cite{MR3683503},
describing the Brauer groups of quiver moduli.
We say that a dimension vector~$\vect{d}$ is \emph{amply stable}
for the stability parameter~$\theta$
if~$\Rep^{\theta\stable}(Q,\vect{d})$ is non-empty,
and the codimension of the~$\theta$-unstable locus in~$\Rep(Q,\vect{d})$ (cf.~\eqref{equation:inclusions-Rep}) is at least~2.
Using this condition, the following description
of the Brauer group is given in \cite[Theorem~4.2]{MR3683503}.
\begin{theorem}[Reineke--Schr\"oer]
  \label{theorem:reineke-schroer}
  Let~$Q$ be a quiver.
  Let~$\vect{d}$ be an amply stable dimension vector
  for the stability parameter~$\theta$.
  Then
  \begin{equation}
    \Br(\modulispace[\theta\stable]{Q,\vect{d}})
    \cong\mathbb{Z}/\gcd(\vect{d})\mathbb{Z},
  \end{equation}
  where the class of the Brauer--Severi scheme~$P_{\vect{n}}$ for~$\vect{n}\neq 0$ is a generator.
\end{theorem}
Assuming \cite[Conjecture~4.3]{MR3683503},
the condition that~$\vect{d}$ is~$\theta$-amply stable can be replaced by
the weaker condition that~$\vect{d}$ admits a~$\theta$-stable representation,
which is simply a non-emptiness condition.
We revisit this conjecture in \cref{subsection:comment-on-reineke-schroer-conjecture}.

\paragraph{Another generator for the Brauer group}
For the computation in \cref{section:main}
we need to use another incarnation of the generator of the Brauer group
in \cref{theorem:reineke-schroer}.
In \eqref{equation:3-moduli-spaces},
the moduli space of stable representations is constructed using geometric invariant theory as
\begin{equation}
  \modulispace[\theta\stable]{Q,\vect{d}}
  \cong
  \Rep^{\theta\stable}(Q,\vect{d})/\!/\GL_{\vect{d}}.
\end{equation}
However, the group~$\PGL_{\vect{d}}$ also acts on~$\Rep^{\theta\stable}(Q,\vect{d})$,
with the same GIT quotient,
and this action is now \emph{free},
so we can in fact construct the natural~$\Gm$-gerbe
\begin{equation}
  \label{equation:Gm-gerbe}
  [\Rep^{\theta\stable}(Q,\vect{d})/\GL_{\vect{d}}]
  \to
  [\Rep^{\theta\stable}(Q,\vect{d})/\PGL_{\vect{d}}]
  \cong
  \modulispace[\theta\stable]{Q,\vect{d}}.
\end{equation}
The following lemma compares the Brauer class
attached to \eqref{equation:Gm-gerbe}
to the one attached to~$[P_{\vect{n}}]$ as in \cref{theorem:reineke-schroer}.
\begin{lemma}
  \label{lemma:brauer-class-comparison}
  The class~$[P_{\vect{n}}]$ in~$\Br(\modulispace[\theta\stable]{Q,\vect{d}})$
  is
  the class of the~$\Gm$-gerbe~\eqref{equation:Gm-gerbe}.
\end{lemma}

\begin{proof}
  The class~$[P_{\vect{n}}]$
  is the obstruction to the existence of a universal representation
  by (the proof of) \cite[Proposition~3.8]{MR2511752}
  or \cite[Theorems~4.2~and~4.4]{MR3683503}.
  Because of the \'etale-local decomposition
  of the universal representation into a direct sum over the vertices of~$Q_0$,
  we see that this obstruction is induced by the class of a~$\PGL_{\vect{d}}$-torsor.
  The obstruction measures when the universal representation,
  which always exists on~$\Rep^{\theta\stable}(Q,\vect{d})$
  and because it is~$\GL_{\vect{d}}$-equivariant
  also on~$[\Rep^{\theta\stable}(Q,\vect{d})/\GL_{\vect{d}}]$,
  descends to~$\modulispace[\theta\stable]{Q,\vect{d}}$
  in the setup of the morphism~\eqref{equation:Gm-gerbe}.

  On the other hand,
  the class of the~$\Gm$-gerbe~\eqref{equation:Gm-gerbe}
  is induced by the class of the~$\PGL_{\vect{d}}$-torsor
  given by~$\Rep^{\theta\stable}(Q,\vect{d})\to\allowbreak\modulispace[\theta\stable]{Q,\vect{d}}$.
  The class of the~$\Gm$-gerbe
  is the obstruction to
  lifting this torsor to a~$\GL_{\vect{d}}$-torsor.
  This lifting is equivalent to the existence of a universal representation
  on~$\modulispace[\theta\stable]{Q,\vect{d}}$,
  because the universal representation on~$\Rep(Q,\vect{d})$
  is given by trivial vector bundles
  with the~$\GL_{\vect{d}}$-action
  that also encodes the quotient.
\end{proof}

\section{Proof of the main result}
\label{section:main}
We assume in this section that~$Q$ is acyclic,
that~$\theta$ is a stability parameter,
and that~$\vect{d}$ is a dimension vector
with~$\theta(\vect{d})=0$
for which~$\vect{d}$ is~$\theta$-amply stable.

We will abbreviate the moduli spaces and stacks introduced in \cref{subsection:recollection} to
\begin{equation}
  \begin{aligned}
    M^\circ&\colonequals\modulispace[\theta\stable]{Q,\vect{d}} \\
    M&\colonequals\modulispace[\theta\semistable]{Q,\vect{d}} \\
    \mathcal{M}
    &=\mathcal{M}^{\theta\semistable}(Q,\vect{d})
    \colonequals[\Rep^{\theta\semistable}(Q,\vect{d})/\GL_{\vect{d}}] \\
    \mathcal{M}^\rig
    &\colonequals\mathcal{M}^{\theta\semistable}(Q,\vect{d})^{\rig}
    =[\Rep^{\theta\semistable}(Q,\vect{d})/\PGL_{\vect{d}}]
  \end{aligned}
\end{equation}
so that we have the commutative diagram
\begin{equation}
  \begin{tikzcd}
    & \widetilde{M} \arrow[d, "\pi"] \\
    M^\circ \arrow[r, hook, "i"] \arrow[ru, hook, "j"] & M & \mathcal{M}^\rig \arrow[l] & \mathcal{M} \arrow[l, "r", swap]
  \end{tikzcd}
\end{equation}
where~$r$ is the partial rigidification removing the common stabiliser~$\Gm$ \cite[Theorem~5.1.5]{MR2007376},
$i$ is the inclusion of the stable locus in the moduli space of~$\theta$-semistable representations,
$\pi$ is any resolution of singularities (which is an isomorphism on the smooth locus)
and~$j$ is the induced inclusion of the stable locus.

Note that~$M^\circ\cong[\Rep^{\theta\stable}(Q,\vect{d})/\PGL_{\vect{d}}]$
as in \eqref{equation:Gm-gerbe}.

The morphism
\begin{equation}
  \label{equation:Gm-gerbe-semistable}
  r\colon\mathcal{M}\to\mathcal{M}^\rig
\end{equation}
is by construction a~$\Gm$-gerbe,
which restricts to the~$\Gm$-gerbe \eqref{equation:Gm-gerbe}.
Unlike the latter,
the codomain in \eqref{equation:Gm-gerbe-semistable}
is not a variety,
but a genuine Artin stack,
at least when properly semistable representations exist.
We will denote the associated class by
\begin{equation}
  \alpha\in\HH_\et^2(\mathcal{M}^\rig,\Gm),
\end{equation}
so that by construction~$\alpha|_{M^\circ}$
is the class of the~$\Gm$-gerbe in \eqref{equation:Gm-gerbe}.

\paragraph{Brauer classes on the stable locus}
Because the global dimension of~$\operatorname{rep}Q$ is~1,
we have an inclusion
\begin{equation}
  \modulispace[\theta\stable]{Q,\vect{d}}
  \subseteq
  \modulispace[\theta\semistable]{Q,\vect{d}}_\smooth
\end{equation}
of the stable locus into the smooth locus.
For moduli of vector bundles this is in fact an equality
\cite[Theorem~4]{MR0242185}
(except when~$g=r=2$),
but for quiver moduli it can be a strict inclusion\footnote{
  If~$Q$ is the 3-Kronecker quiver
  then the moduli space of semistable representations of dimension vector~$(2,2)$
  is all of~$\mathbb{P}^5$.
}.

Because~$\pi\colon\widetilde{M}\to M$
is an isomorphism over the smooth locus,
it is in particular an isomorphism over the stable locus.
Hence, by \cite[Corollaire~1.8]{MR0244270}
the open immersion~$j\colon M^\circ\hookrightarrow\widetilde{M}$
induces the \emph{injective} restriction morphism
\begin{equation}
  \label{equation:restriction-along-j}
  j^*\colon\Br(\widetilde{M})\hookrightarrow\Br(M^\circ)\cong\mathbb{Z}/\gcd(\vect{d})\mathbb{Z},
\end{equation}
analogous to \cite[Lemma~2.4]{MR2925603}.
For the latter isomorphism we use the assumption that~$\vect{d}$ is~$\theta$-amply stable,
so that we can apply \cref{theorem:reineke-schroer}.
Note that we can, and will,
assume that~$\gcd(\vect{d})\geq 2$,
as there is nothing to check for \cref{theorem:main} if~$\gcd(\vect{d})=1$.

This allows us to recast the proof of \cref{theorem:main}
into verifying the equivalent characterisation provided by the following lemma,
which is analogous to \cite[Statement~A]{MR2925603}.
\begin{lemma}
  \label{lemma:equivalent-characterization-main-theorem}
  Assume that~$\vect{d}$ is~$\theta$-amply stable.
  \Cref{theorem:main} is equivalent to the claim that
  the following are equivalent
  for~$\ell\in\mathbb{Z}$:
  \begin{enumerate}
    \item\label{enumerate:first-characterization}
      $\ell\alpha|_{M^\circ}\in\image(j^*\colon\Br(\widetilde{M})\hookrightarrow\Br(M^\circ))$,
      i.e., there exists~$\beta\in\Br(\widetilde{M})$
      such that~$j^*(\beta)=\ell\alpha|_{M^\circ}$;
    \item\label{enumerate:second-characterization}
      $\gcd(\vect{d})\mid \ell$, i.e.,~$\ell\alpha|_{M^\circ}=0$ in~$\Br(M^\circ)$.
  \end{enumerate}
\end{lemma}

\paragraph{Construction of a morphism from $\Spec K$}
Let~$K=\field(x,y)$.
We want to construct a~$K$-point of the strictly semistable locus in~$\mathcal{M}^\rig$,
with certain properties to be specified later.
This will require two ingredients:
the construction of a twisted vector bundle,
and the construction of a certain (semi)stable representation.

The first ingredient is given by
the following lemma,
corresponding to (the proof of) \cite[Lemma~2.3]{MR2925603}.
\begin{lemma}
  \label{lemma:unramified-multiple}
  Let~$n\geq 2$.
  Let~$\nu\colon K\to\mathbb{Z}\cup\{\infty\}$
  be the discrete valuation given by the prime ideal~$(x)$ in~$\field[x,y]$,
  with discrete valuation ring~$R_\nu$.
  Let~$\zeta$ be a primitive~$n$th root of unity,
  and let~$D\colonequals(x,y)_\zeta$ be the cyclic algebra.
  The following are equivalent for an integer~$\ell\in\mathbb{Z}$:
  \begin{enumerate}
    \item $\ell[D]\in\Br(K)$ is unramified at~$\nu$,
      i.e.~$\ell[D]\in\Br(R_\nu)\subseteq\Br(K)$;
    \item $n\mid\ell$.
  \end{enumerate}
\end{lemma}
The Brauer class~$[D]\in\Br(K)\cong\HH_\et^2(\Spec K,\Gm)$
of the cyclic algebra in \cref{lemma:unramified-multiple}
also corresponds to a~$\Gm$-gerbe
\begin{equation}
  \label{equation:gerbe-T}
  \mathcal{T}\to\Spec K,
\end{equation}
on which there exists a~1-twisted vector bundle~$\mathcal{E}$ of rank~$n$
such that~$D\cong\sheafEnd(\mathcal{E})$
(see \cite[\S2.13]{gabber-de-jong} for both claims).

The next ingredient is automatic in the case of vector bundles on a curve,
cf.~\cite[\S2 (iv)]{MR2925603},
but as explained in \cref{subsection:existence-stable}
it is not automatic for moduli of quiver representations.

\begin{lemma}
  \label{lemma:special-semistable-representation}
  Let~$Q$ be a quiver.
  Let~$\vect{d}$ be a dimension vector with~$\gcd(\vect{d})\geq 2$,
  and~$\theta$ a stability parameter
  for which~$\theta(\vect{d})=0$.
  Assume that~$M^\circ$ is non-empty,
  i.e.,
  there exists a~$\theta$-stable representation.
  Then there exists a~$\theta$-stable representation~$V$
  of dimension vector~$\vect{d}/\gcd(\vect{d})$
  such that~$V^{\oplus\gcd(\vect{d})}$ defines a
  strictly~$\theta$-semistable representation
  of dimension vector~$\vect{d}$
  and thus a point~$z\in\mathcal{M}^\rig(\field)$,
  and by abuse of notation,
  a point~$z\in M(\field)$.
\end{lemma}

\begin{proof}
  This follows from \cref{proposition:fractions},
  with~$n=\gcd(\vect{d})$ and dimension vector~$\vect{d}/\gcd(\vect{d})$.
  It is strictly semistable,
  because we assumed that~$\gcd(\vect{d})\geq 2$.
\end{proof}

We will now construct a~$K$-valued point~$x$ of~$\mathcal{M}^\rig$,
with the property that
\begin{itemize}
  \item the class of the restriction of the~$\Gm$-gerbe \eqref{equation:Gm-gerbe-semistable} to~$x$
    coincides with
    the class of the cyclic division algebra $D$ constructed in \cref{lemma:unramified-multiple}
    for $n = \gcd(\vect{d})$;
  \item the image of the point~$x$ in $M(K)$ coincides
   with the image of the point~$z$
    constructed in \cref{lemma:special-semistable-representation}
      under the inclusion~$M(\field)\subset M(K)$,
    so that~$x$ determines a point of~$M\setminus M^\circ$.
\end{itemize}
To do this,
we take the~1-twisted vector bundle~$\mathcal{E}$ on the~$\Gm$-gerbe~$\mathcal{T}\to\Spec K$
given by~\cref{lemma:unramified-multiple} with~$n=\gcd(\vect{d})$,
so that~$\mathcal{E}$ has rank~$\gcd(\vect{d})$,
and take the tensor product (over~$\field$) with
the stable representation~$V$ constructed in \cref{lemma:special-semistable-representation}:
\begin{equation}
  \label{equation:W}
  \mathcal{W}\colonequals\mathcal{E}\otimes_{\field} V.
\end{equation}
Then~$\mathcal{W}$ is an~$\mathcal{O}_{\mathcal{T}}Q$-representation,
which defines a family of~$\theta$-semistable representations
parametrised by the~$\Gm$-gerbe~$\mathcal{T}$,
and we obtain a morphism
\begin{equation}
  f\colon\mathcal{T}\to\mathcal{M}.
\end{equation}
The inertial~$\Gm$ on~$\mathcal{T}$ acts on~$\mathcal{E}$
by scalar multiplication on sections,
and hence acts on~$\mathcal{W}=\mathcal{E}\otimes_{\field}V$
by scalar multiplication as well,
because~$V$ is pulled back from~$\Spec\field$
and so carries the trivial action.
This scaling is precisely the action of the central~$\Gm\subseteq\GL_{\vect{d}}$
that the rigidification~$r$ removes.
Therefore the composition~$r\circ f\colon\mathcal{T}\to\mathcal{M}^\rig$
factors uniquely as
\begin{equation}
  \mathcal{T}\to\Spec K\xrightarrow{x}\mathcal{M}^\rig,
\end{equation}
making the diagram
\begin{equation}
  \label{equation:x}
  \begin{tikzcd}
    \mathcal{T} \arrow[r, "f"] \arrow[d] & \mathcal{M} \arrow[d, "r"] \\
    \Spec K \arrow[r, "x"] & \mathcal{M}^\rig
  \end{tikzcd}
\end{equation}
commute.
In \cite{MR2925603} the analogous morphism~$x$
is constructed in the same way in Step~1 of the proof of their main theorem,
and the analogous identification $x^*\alpha=[D]$
is then asserted without comment;
this implicitly relies on the square being 2-cartesian.
We spell out the missing argument here.
Set~$\mathcal{T}'\colonequals\Spec K\times_{\mathcal{M}^\rig}\mathcal{M}$.
The universal property of the fibre product yields a canonical morphism
\begin{equation}
  g\colon\mathcal{T}\to\mathcal{T}'
\end{equation}
over~$\Spec K$.
Since~$r$ is a~$\Gm$-gerbe,
base change makes~$\mathcal{T}'\to\Spec K$
into a~$\Gm$-gerbe by~\cite[\href{https://stacks.math.columbia.edu/tag/06P3}{Tag~06P3}]{stacks-project}.
The same equivariance that produced the descent~$x$ above
makes~$g$ a morphism of~$\Gm$-gerbes over~$\Spec K$,
and any morphism of gerbes with the same band over a common base
is an isomorphism by~\cite[Lemma~12.2.4]{MR3495343}.
Hence~$g$ is an isomorphism, so the square~\eqref{equation:x} is 2-cartesian.
By construction we have the following result.
\begin{proposition}
  \label{proposition:construction-of-point}
  For~$x\in\mathcal{M}^\rig$ as in \eqref{equation:x}
  we have that
  \begin{enumerate}
    \item We have~$x^*(\alpha)=[D]$ in~$\Br(K)$,
    \item the image of $x$ in $M(K)$ is induced by
      the~$\field$-point~$z$ of~$\mathcal{M}^\rig$
      constructed in \cref{lemma:special-semistable-representation}.
  \end{enumerate}
\end{proposition}

\begin{proof}
  The first identity follows from the 2-cartesian square \eqref{equation:x}:
  pulling back the class~$\alpha$ of the~$\Gm$-gerbe~$r\colon\mathcal{M}\to\allowbreak\mathcal{M}^\rig$
  along~$x$ gives the class of the~$\Gm$-gerbe~$\mathcal{T}\to\Spec K$,
  which is~$[D]$ by the construction of~$\mathcal{T}$ in \eqref{equation:gerbe-T}
  through the identification~$\sheafEnd(\mathcal{E})\cong D$.
  The second claim
  follows by the construction of~$\mathcal{W}$ in \eqref{equation:W}:
  after rigidifying away the~$\Gm$-gerbe
  and forgetting the torsor structure,
  only the data of the orbit of the polystable representation~$V^{\oplus\gcd(\vect{d})}$ remains,
  corresponding to the~$\field$-valued point~$z$.
\end{proof}

\paragraph{Comparing Brauer classes in $\Br(K)$}
Consider the~$K$-point~$x$ of~$\mathcal{M}^\rig$
constructed above.
We now establish the analogue of \cite[Lemma~2.2]{MR2925603},
avoiding their (somewhat terse) construction using Bertini arguments,
giving a direct proof in our special setting.
\begin{lemma}
  \label{lemma:curve-through-x}
  For every~$x\in\mathcal{M}^\rig(K)$
  there exists a pointed morphism
  \begin{equation}
    \label{equation:psi}
    \psi\colon (C,y)\to(\mathcal{M}^\rig,x)
  \end{equation}
  from a smooth and irreducible affine curve~$C$ over~$\Spec K$
  such that~$\psi(C)\cap M^\circ\neq\emptyset$.
\end{lemma}

\begin{proof}
  Let~$P\to\Spec K$ be the~$\PGL_{\vect{d},K}$-torsor defined by~$x$,
  with classifying map~$f\colon\Spec K\to \operatorname{B}\PGL_{\vect{d},K}$.
  The linear representation~$\Rep(Q,\vect{d})$ of~$\PGL_{\vect{d}}$
  corresponds to a locally free sheaf~$\mathscr{E}$ on~$\operatorname{B}\PGL_{\vect{d}}$,
  and by~\cite[(14.2.6)]{MR1771927}
  the associated vector bundle~$\mathbb{V}(\mathscr{E})\to\operatorname{B}\PGL_{\vect{d}}$
  pulls back along~$f$ to the twist
  \begin{equation}
    \Rep(Q,\vect{d})_K\times^{\PGL_{\vect{d},K}}P.
  \end{equation}
  Since every vector bundle over the field~$K$ is trivial,
  this twist is isomorphic to an affine space~$\mathbb{A}_K^N$ for some~$N\geq 1$.
  In particular it is (non-canonically) isomorphic to~$\Rep(Q,\vect{d})_K$ itself.
  The reason for passing to the twist
  is that~$x$, which a priori is only a~$K$-point of the stack~$\mathcal{M}^\rig$,
  is encoded by the torsor~$P$ together with a~$\PGL_{\vect{d},K}$-equivariant morphism to the semistable locus,
  and this datum is precisely a~$K$-rational point of the twisted semistable locus.
  Indeed,
  the~$\theta$-(semi)stable loci in~$\Rep(Q,\vect{d})$
  induce non-empty Zariski-open subschemes of the twist,
  and the point~$x$ defines
  a point~$\widetilde{x}\in(\Rep^{\theta\semistable}(Q,\vect{d})_K\times^{\PGL_{\vect{d},K}}P)(K)$.

  The stable locus in the twist is non-empty by assumption,
  and since $K$ is infinite it has a $K$-rational point.
  Let~$\widetilde{s}\in(\Rep^{\theta\stable}(Q,\vect{d})_K\times^{\PGL_{\vect{d},K}}P)(K)$.
  If~$\widetilde{x}$ and~$\widetilde{s}$
  are distinct,
  then let~$L$ be the affine line through them.
  If~$\widetilde{x}$ and~$\widetilde{s}$ coincide,
  then use any affine line~$L$ through~$\widetilde{x}$.
  In both cases,
  we set~$C\colonequals L\cap(\Rep^{\theta\semistable}(Q,\vect{d})_K\times^{\PGL_{\vect{d},K}}P)$.
\end{proof}

With $\psi\colon C\to\mathcal{M}^\rig$ as in \cref{lemma:curve-through-x},
the resolution of singularities~$\pi\colon\widetilde{M}\to M$ is a projective morphism,
and the generic point~$\eta\in C$
maps to~$M^\circ\subset M$ by the moduli-theoretical construction of~$C$,
as~$\Rep^{\theta\stable}(Q,\vect{d})\subset\Rep^{\theta\semistable}(Q,\vect{d})$ is a dense open.
Because $\pi$ is an isomorphism over $M^\circ$,
the composition $C\to\mathcal{M}^\rig\to M$ lifts to a morphism on the dense open of $C$ mapping to $M^\circ$,
and the valuative criterion for properness applied at each closed point of the smooth curve $C$
extends this to a morphism $h\colon C\to\widetilde{M}$.
We thus obtain the commutative diagram
\begin{equation}
  \begin{tikzcd}
    \Spec K \arrow[r, "y"] \arrow[rr, bend left, "x"] & C \arrow[r, "\psi"] \arrow[d, "h"] & \mathcal{M}^\rig \arrow[d] \\
    & \widetilde{M} \arrow[r, "\pi"] & M \\
    & M^\circ \arrow[r, equals] \arrow[u, hook, "j"] & M^\circ. \arrow[u, hook]
  \end{tikzcd}
\end{equation}
Assume that~$\ell\alpha|_{M^\circ}$ extends to~$\widetilde{M}$,
i.e.,
there exists~$\beta\in\Br(\widetilde{M})$ such that~$j^*(\beta)=\ell\alpha|_{M^\circ}$,
as in the first point of \cref{lemma:equivalent-characterization-main-theorem}.
\begin{lemma}
  \label{lemma:psi-ell-alpha-h-beta}
  With the above setup, we have the equality
  \begin{equation}
    \label{equation:psi-ell-alpha-h-beta}
    \psi^*(\ell\alpha)=h^*(\beta)
  \end{equation}
  in~$\Br(C)$.
\end{lemma}
\begin{proof}
  By construction,
  $C$ is an irreducible regular curve.
  The classes~$\psi^*(\ell\alpha)$ and~$h^*(\beta)$
  coincide on the dense open subset~$h^{-1}(M^\circ)\subset C$.
  Since the natural map from the Brauer group of a regular integral curve
  to the Brauer group of its function field is injective,
  they coincide on~$C$.
\end{proof}
The above construction then allows us to make the following comparison.
\begin{lemma}
  \label{lemma:equality}
  With the above setup, we have the equalities
  \begin{equation}
    \ell[D]=x^*(\ell\alpha)=y^*\circ h^*(\beta)
  \end{equation}
  in~$\Br(K)$.
\end{lemma}

\begin{proof}
  The first equality is \cref{proposition:construction-of-point},
  the second follows from restricting~\cref{lemma:psi-ell-alpha-h-beta}
  along~$y$,
  and using that~$\psi\circ y=x$.
\end{proof}

\paragraph{Conclusion of the proof}
We can now conclude in the same way as \cite{MR2925603}.
\begin{proof}[Proof of \cref{theorem:main}]
  We will now prove that the two characterisations
  of \cref{lemma:equivalent-characterization-main-theorem} are equivalent.
  Let~$R_\nu$ be the discrete valuation ring from \cref{lemma:unramified-multiple},
  with fraction field~$K$.
  Applying the valuative criterion for properness to
  \begin{equation}
    \begin{tikzcd}
      \Spec K \arrow[r, "h\circ y"] \arrow[d] & \widetilde{M} \arrow[d] \\
      \Spec R_\nu \arrow[r] \arrow[ru, dashed] & \Spec\field
    \end{tikzcd}
  \end{equation}
  we obtain the dashed morphism.
  This implies that~$y^*\circ h^*(\beta)$
  comes from a Brauer class on~$\Spec R_\nu$,
  and by the equality in \cref{lemma:equality}
  the class~$\ell[D]$ is unramified.
  By \cref{lemma:unramified-multiple}
  this happens if and only if~$\gcd(\vect{d})\mid\ell$.
  We have thus proved the implication from \cref{enumerate:first-characterization}
  to \cref{enumerate:second-characterization}
  in \cref{lemma:equivalent-characterization-main-theorem}.
  The converse is immediate,
  and thus by \cref{lemma:equivalent-characterization-main-theorem}
  we have finished the proof of \cref{theorem:main}
  in the case that~$\vect{d}$ is~$\theta$-amply stable.
\end{proof}

\section{Further remarks}
\label{section:remarks}

\subsection{Remarks on the statement}
We are now in a position to make some remarks on \cref{theorem:biswas-hogadi-holla,theorem:main}.
\begin{remark}
  \label{remark:motivation}
  When~$\gcd(r,\deg\mathcal{L})=1$
  (resp.~$\gcd(\vect{d})=1$ and~$\theta$ is generic)
  these moduli spaces are smooth projective varieties
  that are \emph{rational} \cite{MR1820549} (resp.~\cite{MR1914089}).
  This implies \cref{theorem:main,theorem:biswas-hogadi-holla}
  by the birational invariance of the Brauer group
  for smooth projective varieties
  \cite[Proposition~6.2.7]{MR4304038}.

  When~$\gcd(r,\deg\mathcal{L})\geq 2$
  (resp.~$\gcd(\vect{d})\geq 5$,
  as rationality is already known when~$\gcd(\vect{d})\in\{2,3,4\}$ \cite{MR0540954,MR0563230})
  these moduli spaces are (usually) \emph{singular}
  and their rationality is a major open question.
  \Cref{theorem:main,theorem:biswas-hogadi-holla} can thus be seen as
  the vanishing of an obstruction to (stable) rationality,
  which also motivated the results by Saltman and Le Bruyn--Schofield.
\end{remark}
Note that two conditions in \cref{theorem:main} can be removed a posteriori,
as explained by the following remarks.
\begin{remark}[Acyclicity]
  \label{remark:acyclicity}
  The acyclicity condition in the statement of \cref{theorem:main}
  is to ensure that
  the moduli space~$\modulispace[\theta\semistable]{Q,\vect{d}}$
  is a \emph{projective} variety,
  which is needed for the proof.
  However, a posteriori one can drop the acyclicity condition,
  provided one rephrases the conclusion as
  the vanishing of the unramified Brauer group of the function field,
  using the stable birational equivalence due to Le Bruyn--Schofield \cite[Theorem~3]{MR1048393}:
  the unramified Brauer group is isomorphic to
  the Brauer group of any smooth and proper model of the function field
  \cite[Proposition~6.2.7]{MR4304038}.
\end{remark}

\begin{remark}[Ample stability]
  \label{remark:ample-stability}
  The condition in \cref{theorem:main}
  that~$\vect{d}$ is~$\theta$-amply stable
  can be replaced by asking that~$\vect{d}$ simply admits a~$\theta$-stable representation,
  by the stable birational invariance
  of the Brauer group of the function field
  due to Le Bruyn--Schofield.
  It is necessary in the current proof,
  because of the proof's dependence on
  the Reineke--Schr\"oer description of the Brauer group in \cref{theorem:reineke-schroer}.
  We revisit this aspect in \cref{subsection:comment-on-reineke-schroer-conjecture}.
\end{remark}

Finally,
it is possible to go directly from the vanishing~$\Br(\widetilde{M})=0$
to the vanishing of~$\Br_\ur(\mathrm{C}_n)$,
as proven by Saltman and Colliot-Th\'el\`ene--Sansuc,
without appealing to stable birational invariance of function fields of quiver moduli,
thus truly giving a self-contained algebro-geometric proof that~$\Br_\ur(\mathrm{C}_n)=0$.
This uses a trick that can be found in, e.g., \cite[page~374]{MR2057404}.
\begin{remark}[Bypassing Le Bruyn--Schofield's stable birational equivalence]
  \label{remark:kronecker-to-jordan}
  Let~$\mathrm{L}_m$ be the~$m$-loop (or generalised Jordan) quiver,
  and let~$\mathrm{K}_{m+1}$ be the~$(m+1)$th generalised Kronecker quiver.
  Let~$n\geq1$ be an integer.
  There exists a morphism
  \begin{equation}
    \Rep(\mathrm{L}_m,(n))
    \to
    \Rep(\mathrm{K}_{m+1},(n,n))
    :
    (M_1,\ldots,M_m)
    \mapsto
    (\mathrm{I}_n,M_1,\ldots,M_m),
  \end{equation}
  where each~$M_i$ is an~$n\times n$-matrix.
  This induces an open immersion
  \begin{equation}
    \modulispace[\semisimple]{\mathrm{L}_m,(n)}
    \hookrightarrow
    \modulispace[\theta\semistable]{\mathrm{K}_{m+1},(n,n)},
  \end{equation}
  which for~$n=1$ is simply the inclusion of~$\mathbb{A}^m$ in~$\mathbb{P}^m$,
  where~$\theta$ is a stability parameter in the unique non-trivial chamber.

  For~$m=2$ we have that
  \begin{equation}
    \field(\modulispace[\semisimple]{\mathrm{L}_2,(n)})
    \cong
    \mathrm{C}_n.
  \end{equation}
  On the other hand,
  by \cite[Proposition~6.2]{MR3683503},
  the~3-Kronecker quiver with dimension vector~$(n,n)$
  is~$\theta$-amply stable when~$n\geq 3$,
  so that we can apply \cref{theorem:main}
  to conclude that $\Br(\widetilde{M})=0$.
  The open immersion above together with \cref{proposition:function-fields}
  gives $\field(\widetilde{M})\cong\mathrm{C}_n$,
  so by \cite[Proposition~6.2.7]{MR4304038}
  we conclude that
  \begin{equation}
    \Br_\ur(\mathrm{C}_n)
    =\Br(\widetilde{M})
    =0.
  \end{equation}
  For~$n=1$ there is nothing to check,
  and for~$n=2$ one can use \cite[Proposition~7.1]{MR3683503},
  which directly verifies \cref{theorem:reineke-schroer},
  and the rest of the proof of \cref{theorem:main} is independent of ample stability.
\end{remark}

\subsection{Centres of generic matrices and quiver moduli}
\label{subsection:saltman}
The more geometrically (as opposed to algebraically) inclined reader might wonder how
the object~$\mathrm{C}_n$ appearing in \cref{table:summary}
relates to quiver moduli.

\paragraph{Centres of generic matrices}
We recall the following constructions from \cite[\S2]{MR0812169} and \cite{MR1177334},
see also \cite{MR1958908}.

For~$n,r\geq 1$ we consider the polynomial algebra~$A=\field[x_{i,j,k}\mid i,j=1,\ldots,n,k=1,\ldots,r]$
and we define the generic matrix~$X_k\in\Mat_n(A)$ for~$k=1,\ldots,r$
as~$(X_k)_{i,j}=x_{i,j,k}$.
These matrices generate the \emph{ring of generic matrices}~$\mathrm{R}(\field,n,r)$
as a subring of~$\Mat_n(A)$.
The ring~$\mathrm{R}(\field,n,r)$ is a domain,
whose central localisation
is the \emph{generic division algebra}~$\mathrm{UD}(\field,n,r)$.
It is a division algebra of degree~$n$ over its centre~$\mathrm{Z}(\field,n,r)$.
It is generic
in the sense that every division algebra of degree~$n$ over an extension field of~$\field$
generated by~$r$ elements over its centre
is a specialisation of~$\mathrm{UD}(\field,n,r)$,
see~\cite[Chapter~14]{MR1692654}.
See \cite{MR0224657} for an early reference.

One sees that~$\mathrm{UD}(\field,n,2)$ is a subalgebra of~$\mathrm{UD}(\field,n,r)$,
inducing~$\mathrm{Z}(\field,n,2)\subseteq\mathrm{Z}(\field,n,r)$,
which is a purely transcendental field extension by \cite[Theorem~3.3.31]{MR0576061}.
This shows how the rationality of~$\mathrm{C}_n\colonequals\mathrm{Z}(\field,n,2)$ is the essential case,
for the influential question whether this is a purely transcendental field extension of~$\field$.
The state of the art is summarised in the survey paper \cite{MR1177334}:
it is rational for~$n=2,3,4$ and stably rational for~$n\mid 420$,
and no progress has been made since.

\paragraph{Centres of generic matrices and the generalised Jordan quiver}
There are several viewpoints on the field~$\mathrm{C}_n$,
the one which is relevant for us is explained in \cite[\S3]{MR1177334}.
The following summarises the discussion.
\begin{proposition}
  \label{proposition:c-n-as-invariant-field}
  Let~$\field$ be a field.
  Let~$\GL_n(\field)$ act on~$\Mat_n(\field)\oplus\Mat_n(\field)$ by simultaneous conjugation,
  and let~$\mathrm{P}_n$ be the polynomial~$\field$-algebra which is its coordinate ring.
  Then~$\mathrm{C}_n\cong\field(\mathrm{P}_n^{\GL_n})$.
\end{proposition}
The discussion in \cref{subsection:recollection}
shows that~$\Spec\mathrm{P}_n^{\GL_n}=\modulispace[\semisimple]{\mathrm{L}_2,(n)}$,
where~$\mathrm{L}_2$ is the generalised Jordan quiver with~2~loops.
Observe that~$(n)$ is a Schur root for this quiver,
so that by \cref{proposition:function-fields}
the function field is also the field of invariant rational functions.

\paragraph{Stable birational reductions of function fields}
The Le Bruyn--Schofield result \cite[Theorem~3]{MR1048393}
shows that~$\field(\modulispace[\theta\semistable]{Q,\vect{d}})$
is \emph{stably isomorphic} to~$\mathrm{C}_n$,
where~$n=\gcd(\vect{d})$,
if~$\vect{d}$ is a Schur root
admitting a~$\theta$-stable representation
(so that \cref{proposition:function-fields} applies).
Or, following the previous discussion,
stably isomorphic to~$\field(\modulispace[\semisimple]{\mathrm{L}_2,(n)})$.

This was later upgraded by Schofield \cite{MR1914089},
showing that~$\field(\modulispace[\theta\semistable]{Q,\vect{d}})$
is \emph{isomorphic} to~$\field(\modulispace[\semisimple]{\mathrm{L}_e,(n)})$,
where we take~$n\colonequals\gcd(\vect{d})$
and~$e\colonequals1-\langle\vect{d}/n,\vect{d}/n\rangle$.
Because we are only interested here in a stable birational invariant
(the unramified Brauer group)
this latter upgrade is not necessary for our discussion,
but it is an important part of the state of the art.

\subsection{The Brauer group of the semistable moduli space}
In \cite[Remark~3.1]{MR2925603} it is claimed that the methods used in op.~cit.
show that
the moduli space of \emph{semistable} vector bundles~$\mathrm{M}_C(r,\mathcal{L})$
has trivial Brauer group.
The argument alluded to in \cite[Remark~3.1]{MR2925603} would also prove that
the moduli space of~$\theta$-\emph{semistable} representations~$\modulispace[\theta\semistable]{Q,\vect{d}}$
has trivial Brauer group,
by the methods in \cref{section:main}.

However, we want to point out that there is seemingly an issue with the argument sketched in op.~cit.
It refers to a result that if an integral scheme~$X$ (with function field~$K$) is locally factorial,
then the natural morphism~$\Br(X)\to\Br(K)$ is injective.
But this seems to require \emph{\'etale}-locally factorial
(or geometrically locally factorial),
cf.~\cite[Theorem~3.5.4]{MR4304038} for the statement with this condition,
and \cite[\S7.6 (3)]{MR4304038} for a locally factorial
yet not geometrically locally factorial example
where the injectivity does not hold.
Thus, it is not clear to us whether
one can conclude that~$\Br(\mathrm{M}_C(r,\mathcal{L}))=0$
(and likewise, it seems one cannot conclude~$\Br(\modulispace[\theta\semistable]{Q,\vect{d}})=0$).

\subsection{When stability coincides with semistability}
\label{subsection:comment-on-reineke-schroer-conjecture}
It is conjectured\footnote{
  There is, however, a typo in the statement: the conjecture must be concerned with
  the moduli space of \emph{stable} representations,
  and not the moduli space of semistable representations,
  to be consistent with \cite[Theorems~4.2 and~4.4]{MR3683503}.
}
in \cite[Conjecture~4.3]{MR3683503} that \cref{theorem:reineke-schroer}
holds \emph{without} the condition of ample stability.
It is mentioned that it is not even known to hold for indivisible dimension vectors.
There is, however, the following easy proposition,
which we will prove as a corollary to \cref{theorem:main}.
\begin{proposition}
  \label{proposition:remark-on-conjecture}
  Let~$Q$ be acyclic,
  and let~$\vect{d}$ be a~$\theta$-coprime dimension vector.
  Then~$\vect{d}$ is indivisible,
  and we have that~$\Br(\modulispace[\theta\stable]{Q,\vect{d}})=0$,
  i.e., \cite[Conjecture~4.3]{MR3683503} holds in this case.
\end{proposition}
\vspace{-.5\baselineskip}
\begin{proof}
  By~$\theta$-coprimality
  we have that~$\modulispace[\theta\stable]{Q,\vect{d}}=\modulispace[\theta\semistable]{Q,\vect{d}}$
  is a smooth and projective variety,
  and by \cite[Corollary~2]{MR1048393}
  we have that it is stably birational
  to some quiver moduli space where ample stability does hold,
  e.g.,
  by taking some Kronecker moduli space with indivisible dimension vector \cite[Proposition~6.2]{MR3683503}.
  We apply \cref{theorem:main} to this moduli space,
  so by stable birational invariance of the Brauer group
  for smooth projective varieties
  we then have~$\Br(\modulispace[\theta\stable]{Q,\vect{d}})=\Br(\modulispace[\theta\semistable]{Q,\vect{d}})=0$,
  which proves the statement.
\end{proof}
It is easy to find examples of this phenomenon,
and thus settling more cases of \cite[Conjecture~4.3]{MR3683503}.
\begin{example}
  Let~$Q$ be the 3-vertex quiver
  \begin{equation}
    Q\colon
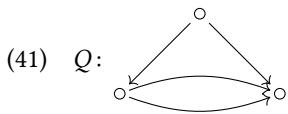
    \begin{tikzpicture}[baseline = -20pt, node distance = 1.5cm]
      \node (1)                      {};
      \node (2) [below left of  = 1] {};
      \node (3) [below right of = 1] {};
      \draw (1) circle (2pt);
      \draw (2) circle (2pt);
      \draw (3) circle (2pt);
      \draw[->] (1) edge (2);
      \draw[->] (1) edge (3);
      \draw[->, bend left = 20]  (2) edge (3);
      \draw[->, bend right = 20] (2) edge (3);
    \end{tikzpicture}
  \end{equation}
  Let~$\vect{d}=(1,1,1)$ be the thin sincere dimension vector,
  and let~$\theta$ be the inner product with~$(2,-1,-1)$.
  Using \textsc{QuiverTools} \cite{MR5076554,quivertools}
  one computes that~$\vect{d}$ is not~$\theta$-amply stable,
  yet~$\modulispace[\theta\stable]{Q,\vect{d}}\cong\mathbb{P}^2$
  (so even isomorphic to a Kronecker moduli space,
  instead of merely stably birational to it as in the proof of \cref{proposition:remark-on-conjecture}),
  whose Brauer group vanishes,
  as predicted by indivisibility.
  One could argue that the data used to define~$\modulispace[\theta\stable]{Q,\vect{d}}$
  was simply poorly chosen,
  with a better choice (using the Kronecker quiver instead)
  immediately implying the desired vanishing.
\end{example}

\printbibliography

\emph{Pieter Belmans}, \url{p.belmans@uu.nl} \\
Mathematical Institute, Utrecht University, Budapestlaan 6, 3584 CD Utrecht, The Netherlands

\emph{Gianni Petrella}, \url{gianni.petrella@uni.lu} \\
Department of Mathematics, Universit\'e de Luxembourg, 6, avenue de la Fonte, L-4364 Esch-sur-Alzette, Luxembourg

\emph{Sebasti\'an Torres}, \url{storresk@udec.cl} \\
Departamento de Matem\'atica, Universidad de Concepci\'on, Casilla 160-C, Concepci\'on, Chile

\end{document}